\documentclass[11pt]{article}
\usepackage{amsmath,amsthm,amscd,latexsym}
\usepackage{amsfonts,pdfsync,color,graphicx}
\usepackage[psamsfonts]{amssymb}
\usepackage{epsfig}
\usepackage{color}
\usepackage[noadjust]{cite}
\usepackage{accents}
\usepackage{listings}
\usepackage[top=1.5in,bottom=1.5in,left=1.3in,right=1.0in]{geometry}

\newtheoremstyle{mystyle}               
{}                
{}                
{}        
{}                
{\bfseries}       
{.}      
{ }      
{}

\newtheorem{theorem}{Theorem}

\newtheorem{corollary}[theorem]{Corollary}

\newtheorem{lemma}[theorem]{Lemma}

\theoremstyle{mystyle}
\newtheorem{example}[theorem]{Example}

\newcommand{\cC}{{\mathcal C}}

\newcommand{\cT}{{\mathcal T}}

\newcommand{\bR}{{\mathbb R}}

\DeclareMathOperator{\dif}{d}

\begin{document}

\title{Existence and multiplicity results for some generalized Hammerstein
equations with a parameter \thanks{%
First author was partially supported by Xunta de Galicia (Spain), project
EM2014/032, AIE Spain and FEDER, grants MTM2013-43014-P, MTM2016-75140-P,
and FPU scholarship, Ministerio de Educaci\'{o}n, Cultura y Deporte, Spain.}}
\date{}
\author{Luc{\' i}a L\'opez-Somoza$^1$\thanks{Corresponding author} and Feliz Minh\'os$^{2,3}$ \\
$^1$ Instituto de Ma\-te\-m\'a\-ti\-cas, Facultade de Matem\'aticas, \\
Universidade de Santiago de Com\-pos\-te\-la, 15782 Santiago de Compostela,\\
Galicia, Spain.\\
lucia.lopez.somoza@usc.es\\
$^2$ Departamento de Matem\'atica, Escola de Ci\^encias e Tecnologia,\\
$^3$ Centro de Investiga\c{c}\~ao em Matem\'atica e Aplica\c{c}\~oes (CIMA),%
\\
Instituto de Investiga\c{c}\~ao e Forma\c{c}\~ao Avan\c{c}ada,\\
Universidade de \'Evora, \'Evora, Portugal.\\
fminhos@uevora.pt}
\maketitle


\begin{abstract}
This paper considers the existence and multiplicity of fixed points for the
integral operator 
\begin{equation*}
{\mathcal{T}}u(t)=\lambda \,\int_{0}^{T}k(t,s)\,f(s,u(s),u^{\prime
}(s),\dots ,u^{(m)}(s))\,\dif s,\quad t\in \lbrack 0,T]\equiv I,
\end{equation*}%
where $\lambda >0$ is a positive parameter, $k:I\times I\rightarrow \mathbb{R%
}$ is a kernel function such that $k\in W^{m,1}\left( I\times I\right) $, $m$
is a positive integer with $m\geq 1$, and $f:I\times 
\mathbb{R}
^{m+1}\rightarrow \lbrack 0,+\infty \lbrack $ is a $L^{1}$-Carath\'{e}odory
function.

The existence of solutions for these Hammerstein equations is obtained by
fixed point index theory on new type of cones. Therefore some assumptions
must hold only for, at least, one of the derivatives of the kernel or, even,
for the kernel, on a subset of the domain. Assuming some asymptotic
conditions on the nonlinearity $f$, we get sufficient conditions for
multiplicity of solutions.

Two examples will illustrate the potentialities of the main results, namely
the fact that the kernel function and/or some derivatives may only be
positive on some subintervals, which can degenerate to a point. Moreover, an
application of our method to general Lidstone problems improves the existent
results on the literature in this field.
\end{abstract}

\section{Introduction}

In this work we will study the existence and multiplicity of fixed points of
the integral operator 
\begin{equation}
{\mathcal{T}}u(t)=\lambda \,\int_{0}^{T}k(t,s)\,f(s,u(s),u^{\prime
}(s),\dots ,u^{(m)}(s))\dif s,\quad t\in \lbrack 0,T]\equiv I,  \label{e-Int}
\end{equation}%
where $\lambda >0$ is a positive parameter, $k:I\times I\rightarrow \mathbb{R%
}$ is a kernel function such that $k\in W^{m,1}\left( I\times I\right) $, $m$
is a positive integer with $m\geq 1$, and $f:I\times 
\mathbb{R}
^{m+1}\rightarrow \lbrack 0,+\infty \lbrack $ is a $L^{1}$-Carath\'{e}odory
function.

The solvability of these type of integral equations, known as Hammerstein
equations (see \cite{Hamm}), has been considered by many authors. In fact
they have become a generalization of differential equations and boundary
value problems and a main field for applications of methods and techniques
of nonlinear analysis, as it can be seen, for instance, in \cite{Ben,
DF+GI+OR, gijwjmaa, GuoSunZhao, InfPietTojo, InfWebb, JiangZhai, MinSousa}.

In \cite{AC+LS+FM}, the authors consider a third order three-point boundary
value problem, whose solutions are the fixed points of the integral operator 
\begin{equation*}
Tu(t)=\lambda \,\int_{0}^{1}G(t,s)\,f(s,u(s),u^{\prime }(s),u^{\prime \prime
}(s))\,\dif s,\quad t\in \lbrack 0,1],
\end{equation*}%
where $G(t,s)$ is an explicit Green's function, verifying some adequate
properties such that $G(t,s)$ and $\frac{\partial \,G}{\partial \,t}(t,s)$
are bounded and non-negative in the square $[0,1]\times \lbrack 0,1],$ but $%
\frac{\partial ^{2}\,G}{\partial \,t^{2}}(t,s)$ could change sign, being
non-negative in a subset of the square.

In \cite{GF+LK+FM}, it is studied a generalized Hammerstein equation%
\begin{equation}
u(t)=\int_{0}^{1}k(t,s)\text{ }g(s)\text{ }f(s,u(s),u^{\prime }(s),\dots
,u^{(m)}(s))\,\dif s,  \label{Eq}
\end{equation}%
with $k:[0,1]^{2}\rightarrow \mathbb{R}$ a kernel function where, in short, $%
k\in W^{m,1}\left( [0,1]^{2}\right) $, $m\geq 1$ is integer, $%
g:[0,1]\rightarrow \lbrack 0,\infty )$ in almost everywhere $t\in \lbrack
0,1]$ non-negative, and $f:[0,1]\times \mathbb{R}^{m+1}\rightarrow \lbrack
0,\infty )$ is a $L^{\infty }-$Carath\'{e}odory function. Moreover, the
kernel $k(t,s)$ and its derivatives $\frac{\partial ^{i}k}{\partial t^{i}}%
(t,s),$ for $i=1,\dots ,m,$ are bounded and non-negative on the square $%
[0,1]\times \lbrack 0,1].$

Our work generalize the existing results in the literature introducing a new
type of cone, 
\begin{equation*}
K=\left\{ 
\begin{split}
u\in \mathcal{C}^{m}(I,\mathbb{R}):\ & u^{(i)}(t)\geq 0,\,t\in \lbrack
m_{i},n_{i}],\,i\in J; \\[3pt]
& \displaystyle\min_{t\in \left[ a_{j},b_{j}\right] }u^{(j)}(t)\geq \xi
_{j}\ \Vert u^{(j)}\Vert _{\left[ c_{j},d_{j}\right] },j\in J_{1}
\end{split}%
\right\} ,
\end{equation*}
where the non-negativeness of the functions may happen only on a
subinterval, possibly degenerate, that is reduced to a point, and, as $%
J_{1}\subset J$, $J_{1}\neq \varnothing ,$ the second property can hold,
locally, only for a restrict number of derivatives, including the function
itself. In this way, it is not required, as it was usual, that $k(t,s)$ and $%
\frac{\partial ^{i}k}{\partial t^{i}}(t,s)$ have constant sign on the
square. Another important novelty is given by $(H_{4})$, where the bounds
must hold only for, at least, one of the derivatives of the kernel or, even,
for the kernel, on a subset of the domain. Assuming some asymptotic
conditions on the nonlinearity $f$\ and applying index theory, we obtain
sufficient conditions for multiplicity of solutions, more precisely, for the
existence of two or more solutions.

Moreover, the application in last section contains new sufficient conditions
for the solvability of $2n$-$th$ order Lidstone problems. In fact, our
method allows that the nonlinearities may depend on derivatives of even and
odd order, which is new in the literature on this type of problems, as it
can be seen, for instance, in \cite{Davis, Marius, Wang, Z+L}. In this way,
our results fill some gaps and improve the study of Lidstone and
complementary Lidstone problems.

The paper is organized in the following way: Section 2 contains the main
assumptions, the definition of the new cone and some properties on the
integral operator. In section 3, the existence results are obtained with
several asymptotic assumptions on $f$ of the sublinear or superlinear type,
near $0$ or $+\infty .$ Section 4 presents existence and multiplicity
results applying fixed point index theory. Section 5 has two examples to
illustrate our main results and, moreover, to emphasize the importance that $%
(H_{4})$ holds only for some derivatives and that the subsets could be
reduced to a point. Last section contains an application to $2n$-$th$ order
Lidstone problems, which allows the dependence of the nonlineariry on odd
and even derivatives.

\section{Hypothesis and auxiliary results}

Let's consider $E=\mathcal{C}^m(I,\,\mathbb{R})$ equipped with the norm 
\begin{equation*}
\|u\|=\max \{\|u^{(i)}\|_{\infty}, \ i\in J \},
\end{equation*}
where $\|v\|_{\infty}=\underset{t\in I}{\sup } \, |v(t)|$ and $%
J\equiv\{0,1,\dots ,m\}$. It is very well-known that $(E,\,\|\cdot\|)$ is a
Banach space.

Throughout the paper we will make the following assumptions:

\begin{itemize}
\item[$(H_{1})$] The kernel function $k\colon I\times I\rightarrow {\mathbb{R%
}}$ is such that $k\in W^{m,1}(I\times I)$, with $m\geq 1$. Moreover, for $i=0,\dots,m-1$, it holds that for every $\varepsilon>0$ and every fixed $\tau \in I$, there exists some $\delta>0$ such that $|t-\tau|<\delta$ implies that 
\begin{equation*}
\left\vert \frac{\partial ^{i}k}{\partial
\,t^{i}}(t,s)-\frac{\partial ^{i}k}{\partial \,t^{i}}(\tau ,s)\right\vert
<\varepsilon \ \text{for a.\,e. } s\in I.
\end{equation*}

Finally, for the $m$-th derivative of the kernel, it holds that 
for every $\varepsilon>0$ and every fixed $\tau \in I$, there exist a set $Z_\tau \in I$ with measure equal to zero and some $\delta>0$ such that $|t-\tau|<\delta$ implies that 
\begin{equation*}
\left\vert \frac{\partial ^{m}k}{\partial
	\,t^{m}}(t,s)-\frac{\partial ^{m}k}{\partial \,t^{m}}(\tau ,s)\right\vert
<\varepsilon, \  \forall\, s\in I\setminus Z_\tau \text{ such that } s<\min\{t,\tau\} \text{ or } s>\max\{t,\tau\}.
\end{equation*}

\item[$(H_{2})$] For each $i\in J_0 \subset J$, $J_0\neq \varnothing$, there
exists a subinterval $[m_{i},n_{i}]$ such that 
\begin{equation*}
\frac{\partial ^{i}k}{\partial t^{i}}(t,s)\geq 0\ \text{ for all }t\in
\lbrack m_{i},n_{i}],\ s\in I.
\end{equation*}
It is possible that the interval is degenerated, that is, $m_i=n_i$.

\item[$(H_{3})$] For $i\in J$, there are positive functions $h_{i}\in
L^{1}(I)$ such that 
\begin{equation*}
\left\vert \frac{\partial ^{i}k}{\partial t^{i}}(t,s)\right\vert \leq
h_{i}(s)\ \text{ for }t\in I \text{ and a.\,e. } s\in I.
\end{equation*}

\item[$(H_{4})$] For each $j\in J_{1}\subset J_0$, $J_{1}\neq \varnothing ,$
there exist subintervals $[a_{j},b_{j}]$ and $[c_{j},d_{j}]$, positive
functions $\phi _{j}\colon I\rightarrow \lbrack 0,\infty )$ and constants $%
\xi _{j}\in (0,1)$ such that 
\begin{equation*}
\left\vert \frac{\partial ^{j}k}{\partial t^{j}}(t,s)\right\vert \leq \phi
_{j}(s)\ \text{ for }t\in \lbrack c_{j},d_{j}] \text{ and a.\,e. } s\in I,
\end{equation*}%
and 
\begin{equation*}
\frac{\partial ^{j}k}{\partial t^{j}}(t,s)\geq \xi _{j}\,\phi _{j}(s)\ \text{
for }t\in \lbrack a_{j},b_{j}] \text{ and a.\,e. } s\in I.
\end{equation*}

Moreover, $\phi _{j}\in L^{1}(I)$ satisfies 
\begin{equation*}
\int_{a_{j}}^{b_{j}}\phi _{j}(s)\,\dif s>0.
\end{equation*}

\item[$(H_{5})$] There exists $i_{0}\in J_0$ such that either $%
[c_{i_{0}},d_{i_{0}}]\equiv I$ or $[m_{i_{0}},n_{i_{0}}]\equiv I$ and,
moreover, $\{0,1,\dots,i_0\}\subset J_0$.

\item[$(H_{6})$] The nonlinearity $f\colon I\times \mathbb{R}%
^{m+1}\rightarrow \lbrack 0,\infty )$ satisfies $L^{1}$-Carath\'{e}odory
conditions, that is,

\begin{itemize}
\item $f(\cdot ,x_{0},\dots ,x_{m})$ is measurable for each $(x_{0},\dots
,x_{m})$ fixed.

\item $f(t,\cdot ,\cdots ,\cdot )$ is continuous for a.\thinspace e. $t\in I$%
.

\item For each $r>0$ there exists $\varphi _{r}\in L^{1}(I)$ such that 
\begin{equation*}
f(t,x_{0},\dots ,x_{m})\leq \varphi _{r}(t),\quad \forall \,(x_{0},\dots
,x_{m})\in (-r,r)^{m+1},\ \text{ a.\thinspace e.}\ t\in I.
\end{equation*}
\end{itemize}

\item[$(H_{7})$] Functions $h_i$ defined in $(H_3)$ and $\varphi_r$ defined
in $(H_6)$ are such that $h_i \, \varphi_r \in L^1(I)$ for every $i\in J$
and $r>0$.
\end{itemize}

We will look for fixed points of operator $\mathcal{T}$ on a suitable cone
on the Banach space $E$. We recall that a cone $K$ is a closed and convex
subset of $E$ satisfying the two following properties:

\begin{itemize}
\item If $x \in K$, then $\lambda\,x\in K$ for all $\lambda\ge 0$.

\item $K\cap(-K)=\{0\}$.
\end{itemize}

Now, taking into account the properties satisfied by the kernel $k$, we
define 
\begin{equation*}
K=\left\{ 
\begin{split}
u\in \mathcal{C}^{m}(I,\mathbb{R}):\ & u^{(i)}(t)\geq 0,\ t\in \lbrack
m_{i},n_{i}],\ i\in J_0; \\[3pt]
& \displaystyle\min_{t\in \left[ a_{j},b_{j}\right] }u^{(j)}(t)\geq \xi
_{j}\ \Vert u^{(j)}\Vert _{\left[ c_{j},d_{j}\right] }, \ j\in J_{1}
\end{split}
\right\},
\end{equation*}%
where 
\begin{equation*}
\left\Vert u^{(j)}\right\Vert _{\left[ c_{j},d_{j}\right] }\colon
=\max_{t\in \left[ c_{j},d_{j}\right] }\left\vert u^{(j)}(t)\right\vert .
\end{equation*}

\begin{lemma}
Hypothesis $(H_{5})$ warrants that $K$ is a cone in $E$.
\end{lemma}

\begin{proof}
We need to verify that $K$ satisfies the two properties which characterize
cones in a Banach space. First of all, from the definition of $K$, it is
trivial to check that if $x \in K$, then $\lambda \, x\in K$ for all $%
\lambda\ge 0$.

Now, to prove that $K\cap (-K)=\{0\}$, we will distinguish between two
different cases:

\begin{itemize}
\item[(I)] There exists $i_0\in J_0$ such that $[m_{i_0},n_{i_0}]\equiv I$.

Suppose that $u,\,-u\in K$. Then $u^{(i_{0})}(t)\geq 0$ and $%
-u^{(i_{0})}(t)\geq 0$ for all $t\in I$, which implies that $%
u^{(i_{0})}\equiv 0$ on $I$. If $i_{0}\geq 1,$ $u^{(i_{0}-1)}$ is constant
on $I$.

Now, we have that $u^{(i_{0}-1)}(t)\geq 0$ and $-u^{(i_{0}-1)}(t)\geq 0$ for
all $t\in \lbrack c_{i_{0}-1},d_{i_{0}-1}]$, that is $u^{(i_{0}-1)}\equiv 0$
on $[c_{i_{0}-1},d_{i_{0}-1}]$. Then, since $u^{(i_{0}-1)}$ is constant on $%
I $, we deduce that $u^{(i_{0}-1)}\equiv 0$ on $I$.

Using the same argument repeatedly we conclude that $u\equiv 0$ on $I$. In
this way, we have proved that $K\cap (-K)=\{0\}$.

\item[(II)] There exists $i_{0}\in J_0$ such that $[c_{i_{0}},d_{i_{0}}]\equiv
I$.

Suppose again that $u,\,-u\in K$. Then, from the fact that 
\begin{equation*}
\min_{t\in \left[ a_{i_{0}},b_{i_{0}}\right] }u^{(i_{0})}(t)\geq \xi
_{i_{0}}\ \Vert u^{(i_{0})}\Vert _{I}\ \text{ and }\ \min_{t\in \left[
a_{i_{0}},b_{i_{0}}\right] }\left( -u^{(i_{0})}(t)\right) \geq \xi _{i_{0}}\
\Vert u^{(i_{0})}\Vert _{I},
\end{equation*}%
it is deduced that $\Vert u^{(i_{0})}\Vert _{I}=0$, which implies that $%
u^{(i_{0})}\equiv 0$ on $I$. Now, following the same arguments than in
Case (I), we deduce the result.
\end{itemize}
\end{proof}

In next section, considering some additional properties on the function $f$,
we will ensure the existence of fixed points of operator ${\mathcal{T}}$.
However, before doing that, we obtain some previous technical results.

\begin{lemma}
\label{L:complet_cont} If $(H_{1})-(H_{7})$ hold, then ${\mathcal{T}}\colon
K\rightarrow K$ is a completely continuous operator.
\end{lemma}

\begin{proof}
We divide the proof into several steps.

\vspace*{0.2cm} \textbf{Step 1.} $\cT$ is well defined in $K$.

\vspace*{0.1cm} Let $u\in K$. The proof that $\cT u\in \cC^m(I,\bR)$ follow standard techniques and so we omit it.

%
%

%
We will prove that $\cT u\in K$.

It is obvious that, for $i\in J_0$, $({\mathcal{T}}u)^{(i)}(t)\geq 0$ for all $t\in \lbrack m_{i},n_{i}]$.

Moreover, for $j\in J_{1}$ and $t\in \lbrack c_{j},d_{j}]$, we have that 
\begin{eqnarray*}
\left\vert ({\mathcal{T}}u)^{(j)}(t)\right\vert &\leq &\lambda
\,\int_{0}^{T}\left\vert \frac{\partial ^{j}k}{\partial \,t^{j}}%
(t,s)\right\vert \,f(s,u(s),\dots ,u^{(m)}(s))\,\dif s \\
&\leq &\lambda \,\int_{0}^{T}\phi _{j}(s)\,f(s,u(s),\dots ,u^{(m)}(s))\,\dif s,
\end{eqnarray*}%
and, taking the supremum for $t\in \lbrack c_{j},d_{j}]$, we deduce that 
\begin{equation*}
\left\Vert ({\mathcal{T}}u)^{(j)}\right\Vert _{[c_{j},d_{j}]}\leq \lambda
\,\int_{0}^{T}\phi _{j}(s)\,f(s,u(s),\dots ,u^{(m)}(s))\,\dif s.
\end{equation*}%
On the other hand, for $t\in \left[ a_{j},b_{j}\right] $, we have 
\begin{equation*}
\begin{split}
({\mathcal{T}}u)^{(j)}(t)& =\lambda \,\int_{0}^{T}\frac{\partial ^{j}k}{%
\partial \,t^{j}}(t,s)\,f(s,u(s),\dots ,u^{(m)}(s))\,\dif s \\
& \geq \lambda \,\int_{0}^{T}\xi _{j}\,\phi _{j}(s)\,f(s,u(s),\dots
,u^{(m)}(s))\,\dif s\geq \xi _{j}\,\left\Vert ({\mathcal{T}}u)^{(j)}\right\Vert
_{[c_{j},d_{j}]}
\end{split}%
\end{equation*}%
and we deduce that 
\begin{equation*}
\min_{t\in \left[ a_{j},b_{j}\right] }({\mathcal{T}}u)^{(j)}(t)\geq \xi
_{j}\,\left\Vert ({\mathcal{T}}u)^{(j)}\right\Vert _{[c_{j},d_{j}]}
\end{equation*}%
for $j\in J_{1}$.

Therefore, we can conclude that ${\mathcal{T}}u\in K$.

\vspace*{0.2cm} \textbf{Step 2.} $\cT$ is continuous in $\cC^m(I,\bR)$.

This part also follows standard techniques.

\vspace*{0.2cm} \textbf{Step 3.} $\cT$ is a compact operator.

\vspace*{0.1cm} Let's consider 
\begin{equation*}
B=\{u\in E;\ \Vert u\Vert \leq r\}.
\end{equation*}

First, we will prove that $T(B)$ is uniformly bounded in $\mathcal{C}^{m}(I)$%
.

We find the following bounds for $u\in B$ and $i\in J$: 
\begin{equation*}
\begin{split}
\left\Vert ({\mathcal{T}}u)^{(i)}\right\Vert _{\infty }& =\sup_{t\in
I}\left\vert \lambda \,\int_{0}^{T}\frac{\partial ^{i}\,k}{\partial \,t^{i}}%
(t,s)\,f(s,u(s),\dots ,u^{(m)}(s))\,\dif s\right\vert \\[2pt]
& \leq \lambda \,\int_{0}^{T}h_{i}(s)\,f(s,u(s),\dots ,u^{(m)}(s))\,\dif s \leq \lambda \,\int_{0}^{T}h_{i}(s)\,\varphi _{r}(s)\,\dif s:={M}_{i}, 
\end{split}%
\end{equation*}%
with ${M}_{i}>0$. Therefore, 
\begin{equation*}
\Vert {\mathcal{T}}u\Vert \leq \max \{{M}_{i}:\ i\in J\},\quad \forall \,u\in B.
\end{equation*}

\vspace*{0.2cm} Now, we will prove that ${\mathcal{T}}(B)$ is equicontinuous
in $\mathcal{C}^{m}(I)$. Let $t_{2}\in I$ be fixed. Then, for every $\varepsilon>0$, take $\delta>0$ given in $(H_1)$ and for $i=0,\dots,m-1$, it holds that $|t_1-t_2|<\delta$ implies that 
\begin{equation*}\begin{split}
\left\vert ({\mathcal{T}}u)^{(i)}(t_{1})-({\mathcal{T}}u)^{(i)}(t_{2})\right\vert & \leq \lambda \int_{0}^{T}\left\vert \frac{\partial ^{i}\,k}{\partial
	\,t^{i}}(t_{1},s)-\frac{\partial ^{i}\,k}{\partial \,t^{i}}%
(t_{2},s)\right\vert f(s,u(s),\dots ,u^{(m)}(s))\,\dif s \\[2pt]
& \leq \lambda \int_{0}^{T}\left\vert \frac{\partial ^{i}\,k}{\partial \,t^{i}} (t_{1},s) -\frac{\partial^{i}\,k}{\partial \,t^{i}}
(t_{2},s)\right\vert \,\varphi _{r}(s)\,\dif s \le \varepsilon\,\lambda \int_{0}^{T} \varphi _{r}(s)\,\dif s,
\end{split}\end{equation*}
and, since $\varphi_r\in L^1(I)$, it is clear that there exists a positive constant $\kappa_1$ such that 
\[\left\vert ({\mathcal{T}}u)^{(i)}(t_{1})-({\mathcal{T}}u)^{(i)}(t_{2})\right\vert < \kappa_1 \, \varepsilon \]
for all $u\in B$. 

On the other hand, for the $m$-th derivative, for every $\varepsilon>0$, take $\delta>0$ given in $(H_1)$ and $|t_1-t_2|<\delta$, $t_1<t_2$, implies that 
\begin{equation*}\begin{split}
\left\vert ({\mathcal{T}}u)^{(m)}(t_{1})-({\mathcal{T}}u)^{(m)}(t_{2})\right\vert  \leq & \, \lambda \int_{0}^{T}\left\vert \frac{\partial ^{m}\,k}{\partial
	\,t^{m}}(t_{1},s)-\frac{\partial ^{m}\,k}{\partial \,t^{m}}%
(t_{2},s)\right\vert f(s,u(s),\dots ,u^{(m)}(s))\,\dif s \\[2pt]
\leq & \, \lambda \int_{0}^{T}\left\vert \frac{\partial ^{m}\,k}{\partial \,t^{m}} (t_{1},s) -\frac{\partial^{m}\,k}{\partial \,t^{m}}
(t_{2},s)\right\vert \,\varphi _{r}(s)\,\dif s \\[2pt]
= & \, \lambda \int_{0}^{t_1}\left\vert \frac{\partial ^{m}\,k}{\partial \,t^{m}} (t_{1},s) -\frac{\partial^{m}\,k}{\partial \,t^{m}}
(t_{2},s)\right\vert \,\varphi _{r}(s)\,\dif s \\[2pt]
& \, + \lambda \int_{t_1}^{t_2}\left\vert \frac{\partial ^{m}\,k}{\partial \,t^{m}} (t_{1},s) -\frac{\partial^{m}\,k}{\partial \,t^{m}}
(t_{2},s)\right\vert \,\varphi _{r}(s)\,\dif s \\[2pt]
& \, + \lambda \int_{t_2}^{T}\left\vert \frac{\partial ^{m}\,k}{\partial \,t^{m}} (t_{1},s) -\frac{\partial^{m}\,k}{\partial \,t^{m}}
(t_{2},s)\right\vert \,\varphi _{r}(s)\,\dif s.
\end{split}\end{equation*}
From $(H_1)$, it is clear that first and third integrals in last term of previous expression can be arbitrarily small when $|t_1-t_2|<\delta$. Moreover, $\left\vert \frac{\partial ^{m}\,k}{\partial \,t^{m}} (t_{1},\cdot) -\frac{\partial^{m}\,k}{\partial \,t^{m}}
(t_{2},\cdot)\right\vert \,\varphi _{r}(\cdot) \in L^1[t_1,t_2]$ and so there exists some $\delta'>0$ such that
\[\lambda \int_{t_1}^{t_2}\left\vert \frac{\partial ^{m}\,k}{\partial \,t^{m}} (t_{1},s) -\frac{\partial^{m}\,k}{\partial \,t^{m}}
(t_{2},s)\right\vert \,\varphi _{r}(s)\,\dif s <\varepsilon \]
when $|t_1-t_2|<\delta'$.

Therefore it is clear that, for $|t_1-t_2|<\min\{\delta,\delta'\}$, $t_1<t_2$, there exists a positive constant $\kappa_2$ such that
\[\left\vert ({\mathcal{T}}u)^{(m)}(t_{1})-({\mathcal{T}}u)^{(m)}(t_{2})\right\vert< \kappa_2\,\varepsilon\]
for all $u\in B$.

Analogously, when $|t_1-t_2|<\delta$, $t_1>t_2$,  there exists some some positive constant $\kappa_3$ such that
\[\left\vert ({\mathcal{T}}u)^{(m)}(t_{1})-({\mathcal{T}}u)^{(m)}(t_{2})\right\vert< \kappa_3 \,\varepsilon\] 
for all $u\in B$.



We have proved the pointwise equicontinuity on $I$. Moreover, since $I$ is compact, pointwise equicontinuity is equivalent to uniform equicontinuity.


This way, we conclude that ${\mathcal{T}}(B)$ is equicontinuous in $\mathcal{C}^{m}(I)$.

\vspace{0.2cm} As a consequence, by Ascoli-Arzel\`{a} Theorem, we can affirm that ${\mathcal{T}}(B)$ is relatively compact in $\mathcal{C}^{m}(I)$ and so ${\mathcal{T}}$ is a completely continuous operator.
\end{proof}

\section{Main results}

We introduce now the following notation 
\begin{equation*}
\Lambda ^{i}:=\int_{0}^{T}h_{i}(s)\,\dif s,\quad \Lambda
_{i}:=\int_{a_{i}}^{b_{i}}\xi _{i}\,\phi _{i}(s)\,\dif s
\end{equation*}%
and define 
\begin{equation*}
\bar{\Lambda}:=(m+1)\,\max \{\Lambda ^{i}:\ i\in J\}\quad \text{and}\quad 
\underline{{\Lambda }}{:}=\max \{\xi _{i}\,\Lambda _{i}:\ i\in J_1\}.
\end{equation*}

On the other hand, we denote 
\begin{equation*}
f_{0}:=\liminf\limits_{|x_{0}|,\dots ,|x_{m}|\rightarrow 0}\,\min_{t\in I}%
\frac{f(t,x_{0},\dots ,x_{m})}{|x_{0}|+\cdots +|x_{m}|}
\end{equation*}%
and 
\begin{equation*}
f^{\infty }:=\limsup\limits_{|x_{0}|,\dots ,|x_{m}|\rightarrow \infty
}\,\max_{t\in I}\frac{f(t,x_{0},\dots ,x_{m})}{|x_{0}|+\cdots +|x_{m}|}.
\end{equation*}

We will give now our existence result.

\begin{theorem}
\label{T:exist1_orden_n} Assume that hypotheses $(H_{1})-(H_{7})$ hold. If $%
\bar{\Lambda}\,f^{\infty }<\underline{{\Lambda }}\,f_{0}$, then for all 
\begin{equation*}
\lambda \in \left( \frac{1}{\underline{{\Lambda }}\,f_{0}},\frac{1}{\bar{%
\Lambda}\,f^{\infty }}\right)
\end{equation*}%
operator ${\mathcal{T}}$ has a fixed point in the cone $K$, which is not a
trivial solution.
\end{theorem}

\begin{proof}
Let $\lambda \in \left( \frac{1}{\underline{{\Lambda }}\,f_{0}},\frac{1}{%
\bar{\Lambda}\,f^{\infty }}\right) $ and choose $\varepsilon \in \left(
0,f_{0}\right) $ such that 
\begin{equation*}
\frac{1}{\underline{{\Lambda }}\,(f_{0}-\varepsilon )}\leq \lambda \leq 
\frac{1}{\bar{\Lambda}\,(f^{\infty }+\varepsilon )}.
\end{equation*}%
Taking into account the definition of $f_{0}$, we know that there exists $%
\delta _{1}>0$ such that when $\Vert u\Vert \leq \delta _{1}$, 
\begin{equation*}
f(t,u(t),\dots ,u^{(m)}(t))>\,(f_{0}-\varepsilon )\,\left( |u(t)|+\cdots
+|u^{(m)}(t)|\right) ,\quad \forall \,t\in I.
\end{equation*}%
Let 
\begin{equation*}
\Omega _{\delta _{1}}=\{u\in K;\,\Vert u\Vert <\delta _{1}\}
\end{equation*}%
and choose $u\in \partial \,\Omega _{\delta _{1}}$. We will prove that $\mathcal{T}u\not\preceq u$. We have that for $t\in \lbrack a_{i},b_{i}]$ and $j\in J_{1}$,
\begin{equation*}
\begin{split}
(\mathcal{T}u)^{(j)}(t)=& \,\lambda \int_{0}^{T}\frac{\partial ^{j}k}{%
\partial \,t^{j}}(t,s)\,f(s,u(s),\dots ,u^{(m)}(s))\,\dif s \geq 
\lambda
\int_{a_{j}}^{b_{j}}\frac{\partial ^{j}k}{\partial \,t^{j}}%
(t,s)\,f(s,u(s),\dots ,u^{(m)}(s))\,\dif s \\
\geq & \,\lambda \int_{a_{j}}^{b_{j}}\xi _{j}\,\phi _{j}(s)\,f(s,u(s),\dots
,u^{(m)}(s))\,\dif s \\
>& \,\lambda \int_{a_{j}}^{b_{j}}\xi _{j}\,\phi _{j}(s)\,(f_{0}-\varepsilon
)\,\left( |u(s)|+\cdots +|u^{(m)}(s)|\right) \,\dif s \\
\geq & \,\lambda \,(f_{0}-\varepsilon )\,\xi _{j}\,\Vert u^{(j)}\Vert
_{\lbrack a_{j},b_{j}]}\,\int_{a_{j}}^{b_{j}}\xi _{j}\,\phi _{j}(s)\,\dif s \\
=& \,\lambda \,(f_{0}-\varepsilon )\,\xi _{j}\,\Vert u^{(j)}\Vert _{\lbrack
a_{j},b_{j}]}\,\Lambda _{j}\geq \,\lambda \,(f_{0}-\varepsilon )\,\Lambda
_{j}\,\xi _{j}\,u^{(j)}(t).
\end{split}%
\end{equation*}%
As a consequence we have that for some $j\in J_{1}$, $(\mathcal{T}%
u)^{(j)}(t)>u^{(j)}(t)$ for all $t\in \lbrack a_{j},b_{j}]$ and so it is
proved that $\mathcal{T}u\not\preceq u$. We deduce (see \cite[Theorem 2.3.3]%
{GuoLaks}) that 
\begin{equation*}
i_{K}(\mathcal{T},\,\Omega _{\delta_{1}})=0.
\end{equation*}%
On the other hand, due to the definition of $f^{\infty }$, we know that
there exists $\tilde{C}>0$ such that when $\min \left\{
|u^{(i)}(t)|:i\in J\right\} \geq \tilde{C}$, 
\begin{equation*}
f(t,u(t),\dots ,u^{(m)}(t))\leq (f^{\infty }+\varepsilon )\,\left(
|u(t)|+\cdots +|u^{(m)}(t)|\right) \leq (m+1)\,(f^{\infty }+\varepsilon
)\,\Vert u\Vert ,\quad \forall \,t\in I.
\end{equation*}%
Let $C>\{\delta _{1},\,\tilde{C}\}$ and define 
\begin{equation*}
\Omega _{C}=\bigcup\limits_{i=0}^{m}\left\{ u\in K:\,\min_{t\in
I}|u^{(i)}(t)|<C\right\} .
\end{equation*}%
We note that $\Omega _{C}$ is an unbounded subset of the cone $K$.
Because of this, the fixed point index of operator $\mathcal{T}$ with
respect to $\Omega _{C}$, $i_{K}(\mathcal{T},\,\Omega _{C})$, is only defined in the case that the set of fixed points of
operator $\mathcal{T}$ in $\Omega _{C}$, that is, $(I-\mathcal{T}%
)^{-1}(\{0\})\cap \Omega _{C}$, is compact (see \cite{GranasDug}
for the details). We will see that $i_{K}(\mathcal{T},\,\Omega _{C})$ can be defined in this case.

First of all, since $(I-\mathcal{T})$ is a continuous operator, it is
obvious that $(I-\mathcal{T})^{-1}(\{0\})\cap \Omega _{C}$ is
closed.

Moreover, we can assume that $(I-\mathcal{T})^{-1}(\{0\})\cap \Omega
_{C}$ is bounded. Indeed, on the contrary, we would have infinite
fixed points of operator $\mathcal{T}$ on $\Omega _{C}$ and it
would be immediately deduced that $\mathcal{T}$ has an infinite number of
fixed points in the cone $K$. Therefore we may assume that there exists a
constant $M>0$ such that $\Vert u\Vert <M$ for all $u\in \,(I-\mathcal{T}%
)^{-1}(\{0\})\cap \Omega _{C}$.

Finally, it is left to see that $(I-\mathcal{T})^{-1}(\{0\})\cap \Omega
_{C}$ is equicontinuous. This property follows from the fact that $%
(I-\mathcal{T})^{-1}(\{0\})\cap \Omega _{C}$ is bounded. The proof
is totally analogous to Step 3 in the proof of Lemma \ref{L:complet_cont}.

Now, we will calculate $i_{K}(\mathcal{T},\,\Omega _{C})$. In
particular, we will prove that $\Vert \mathcal{T}u\Vert \leq \Vert u\Vert $
for all $u\in \partial \,\Omega _{C}$. Let $u\in \partial \,\Omega
_{C}$, that is, $u\in K$ such that 
\begin{equation*}
\min \left\{ \min_{t\in I}|u^{(i)}(t)|:\ i\in J\right\} =C.
\end{equation*}%
Then, for $i\in J$,
\begin{equation*}
\begin{split}
|(\mathcal{T}u)^{(i)}(t)|\leq & \,\lambda \int_{0}^{T}\left\vert \frac{%
\partial ^{i}k}{\partial \,t^{i}}(t,s)\right\vert \,f(s,u(s),\dots
,u^{(m)}(s))\,\dif s  \leq \lambda \int_{0}^{T}h_{i}(s)\,f(s,u(s),\dots ,u^{(m)}(s))\,\dif s \\[2pt]
\leq & \,(m+1)\,\lambda \int_{0}^{T}h_{i}(s)\,(f^{\infty }+\varepsilon
)\,\Vert u\Vert \,\dif s
= (m+1)\,\lambda \,(f^{\infty }+\varepsilon )\,\Vert u\Vert \,\Lambda
^{i}\\
\leq & \,\lambda \,(f^{\infty }+\varepsilon )\,\Vert u\Vert \,\bar{\Lambda}%
\leq \Vert u\Vert .
\end{split}%
\end{equation*}%
We deduce that 
\begin{equation*}
\Vert \mathcal{T}u\Vert \leq \Vert u\Vert
\end{equation*}%
and as a consequence (\cite[Corollary 7.4]{GranasDug}) we have that 
\begin{equation*}
i_{K}(\mathcal{T},\,\Omega _{C})=1.
\end{equation*}%
Therefore, we conclude that $\mathcal{T}$ has a fixed point in $\bar{\Omega}%
_{C}\setminus \Omega _{\delta _{1}}$.
\end{proof}

Consequently, we obtain the following

\begin{corollary}
\label{cor_exist1} Assume that hypotheses $(H_{1})-(H_{7})$ hold. Then,

\begin{itemize}
\item[(i)] If $f_0=\infty$ and $f^{\infty}=0$, then for all $%
\lambda\in(0,\infty)$, ${\mathcal{T}}$ has a fixed point in the cone.

\item[(ii)] If $f_0=\infty$ and $0<f^{\infty}<\infty$, then for all $%
\lambda\in\left(0,\frac{1}{\bar{\Lambda}\,f^{\infty}}\right)$, ${\mathcal{T}}
$ has a fixed point in the cone.

\item[(iii)] If $0<f_{0}<\infty $ and $f^{\infty }=0$, then for all $\lambda
\in \left( \frac{1}{\underline{{\Lambda }}\,f_{0}},\infty \right) $, ${%
\mathcal{T}}$ has a fixed point in the cone.
\end{itemize}
\end{corollary}

\section{Existence and multiplicity of solutions}

In this section we will use the fixed point index theory to study the
existence of multiple fixed points of operator ${\mathcal{T}}$. In \cite%
{CabInfTojo} the authors apply similar arguments to functional equations
that only depend on the values of $u$. First of all, we compile some
classical results regarding to fixed point index (see \cite{Amann,GuoLaks}
for the details).

\begin{lemma}
Let $D$ be an open bounded set with $D_K=D\cap K\neq \emptyset$ and $\bar{D}%
_K\neq K$. Assume that $F\colon \bar{D}_K\rightarrow K$ is a compact map
such that $x\neq F\,x$ for $x\in\partial D_K$. Then the fixed point index $%
i_K(F,D_K)$ has the following properties:

\begin{itemize}
\item[(1)] If there exists $e\in K\setminus\{0\}$ such that $x\neq
F\,x+\alpha\,e$ for all $x\in\partial D_K$ and all $\alpha>0$, then $%
i_K(F,D_K)=0$.

\item[(2)] If $\mu\,x\neq F\,x$ for all $x\in\partial D_K$ and for every $%
\mu\ge 1$, then $i_K(F,D_K)=1$.

\item[(3)] Let $D^1$ be open in $X$ with $\bar{D}^1\subset D_K$. If $%
i_K(F,D_K)=1$ and $i_K(F,D^1_K)=0$, then $F$ has a fixed point in $%
D_K\setminus \bar{D}^1_K$. The same result holds if $i_K(F,D_K)=0$ and $%
i_K(F,D^1_K)=1$.
\end{itemize}
\end{lemma}

We will define the following sets: 
\begin{equation*}
K_{\rho}=\{u\in K; \, \|u\|<\rho\},
\end{equation*}
\begin{equation*}
V_{\rho}=\left\{u\in K: \, \min_{t\in\left[a_i,b_i\right]} u^{(i)}(t)<\rho,
\, i\in J_2, \ \ \|u^{(i)}\|_\infty<\rho, \, i\in J\setminus J_2\right\},
\end{equation*}
where $J=\{0,\dots,m\}$ and 
\begin{equation*}
J_2=\left\{i\in J: \ [c_i,d_i]\equiv I \right\}.
\end{equation*}

To ensure that the sets $K_\rho$ and $V_\rho$ are not the same, we need to
change condition $(H_5)$ into

\begin{itemize}
\item[$(\widetilde H_5)$] There exists $i_0\in\{0,\dots,m\}$ such that $%
[c_{i_0},d_{i_0}]\equiv I$ and, moreover, $\{0,1,\dots,i_0\}\subset J_0$.
\end{itemize}

In this situation, it is clear that $J_2\neq \emptyset$ and therefore 
\begin{equation*}
K_{\rho}\subsetneq V_{\rho}\subsetneq K_{\frac{\rho}{c}}
\end{equation*}
where 
\begin{equation}  \label{e-c-def}
c=\min\{\xi_i: \ i\in J_2\}.
\end{equation}

Now we will give sufficient conditions under which the index of the previous
sets is either 1 or 0.

\begin{lemma}
\label{L:index1_fp} Let 
\begin{equation*}
\frac{1}{N}=\max\left\{\sup_{t\in I} \int_{0}^{T}\left|\frac{\partial^i k}{%
\partial\,t^i}(t,s)\right|\,\dif s: \ i\in J \right\}
\end{equation*}
and 
\begin{equation*}
f^{\rho}=\sup\left\{\frac{f(t,x_0,\dots,x_m)}{\rho};\ t\in I, \ x_i\in[%
-\rho,\rho], \ i\in J \right\}.
\end{equation*}
If there exists $\rho>0$ such that 
\begin{equation}
\lambda\,\frac{f^{\rho}}{N}<1,  \tag{$I^1_{\rho}$}
\end{equation}
then $i_K({\mathcal{T}},K_{\rho})=1$.
\end{lemma}

\begin{proof}
We will prove that ${\mathcal{T}} u\neq \mu\,u$ for all $u\in\partial
K_{\rho}$ and for every $\mu\ge 1$.

Suppose, on the contrary, that there exist some $u\in\partial K_{\rho}$ and $%
\mu\ge 1$ such that 
\begin{equation*}
\mu\,u^{(i)}(t)=\lambda \int_{0}^{T}\frac{\partial^i k}{\partial\,t^i}%
(t,s)\,f(s,u(s),\dots,u^{(m)}(s))\,\dif s.
\end{equation*}
Taking the supremum for $t\in I$, we obtain that 
\begin{equation*}
\begin{split}
\mu\,\|u^{(i)}\|_{\infty} &\le \lambda\,\sup_{t\in I}\int_{0}^{T} \left|%
\frac{\partial^i k}{\partial\,t^i}(t,s)\right|
\,f(s,u(s),\dots,u^{(m)}(s))\,\dif s \le \lambda\,\rho \, f^{\rho} \, \sup_{t\in
I} \int_{0}^{T}\left|\frac{\partial^i k}{\partial\,t^i}(t,s)\right|\,\dif s \\%
[0.1cm]
&\le \lambda\,\rho \,\frac{f^{\rho}}{N}<\rho.
\end{split}%
\end{equation*}

Consequently, we deduce that 
\begin{equation*}
\mu \,\rho =\mu \,\max \{\Vert u^{(i)}\Vert _{\infty }:\,i\in J\}<\rho ,
\end{equation*}%
which contradicts the assumption that $\mu \geq 1$. Therefore, $i_{K}({%
\mathcal{T}},K_{\rho })=1$.
\end{proof}

\begin{lemma}
\label{L:index0-fpi} For $i\in J_1$, let 
\begin{equation*}
\frac{1}{M_i}=\inf_{t\in\left[a_i,b_i\right]} \int_{a_i}^{b_i}\frac{%
\partial^i k}{\partial\,t^i}(t,s)\,\dif s,
\end{equation*}
and 
\begin{equation*}
f^i_{\rho}=\inf\left\{\frac{f(t,x_0,\dots,x_m)}{\rho}:\ t\in[a_i,b_i], \
x_j\in \left[0,\frac{\rho}{\xi_j}\right], \, j\in J_2, \ x_k\in\left[0,\rho%
\right],\, k\in J\setminus J_2 \right\}.
\end{equation*}
If there exists $\rho>0$ and $i_0\in J_1$ such that 
\begin{equation}
\lambda\,\frac{f^{i_0}_{\rho}}{M_{i_0}}>1,  \tag{$I^0_{\rho}$}
\end{equation}
then $i_K({\mathcal{T}},V_{\rho})=0$.
\end{lemma}

\begin{proof}
We will prove that there exists $e\in K\setminus\{0\}$ such that $u\neq {%
\mathcal{T}} u+\alpha\,e$ for all $u\in\partial V_{\rho}$ and all $\alpha>0$.

Let us take $e(t)=1$ and suppose that there exists some $u\in\partial
V_{\rho}$ and $\alpha>0$ such that $u={\mathcal{T}} u+\alpha$. Then, for $%
t\in\left[a_{i_0},b_{i_0}\right]$, 
\begin{equation*}
\begin{split}
u^{({i_0})}(t)\ge & \,\lambda\int_{0}^{T}\frac{\partial^{i_0} k}{%
\partial\,t^{i_0}}(t,s)\,f(s,u(s),\dots,u^{(m)}(s))\,\dif s \ge
\lambda\int_{a_{i_0}}^{b_{i_0}}\frac{\partial^{i_0} k}{\partial\,t^{i_0}}%
(t,s)\,f(s,u(s),\dots,u^{(m)}(s))\,\dif s \\
\ge & \, \lambda\,\rho\,f^{i_0}_{\rho}\int_{a_{i_0}}^{b_{i_0}}\frac{%
\partial^{i_0} k}{\partial\,t^{i_0}}(t,s)\,\dif s>\rho.
\end{split}%
\end{equation*}

Consequently, $u^{({i_0})}(t)>\rho$ for $t\in\left[a_{i_0},b_{i_0}\right]$,
which is a contradiction. Thus, ${i_K({\mathcal{T}},V_\rho)=0}$.
\end{proof}

Combining the previous lemmas, it is possible to obtain some conditions
under which operator ${\mathcal{T}}$ has multiple fixed points.

\begin{theorem}
\label{T:exist2_orden_n} Assume that conditions $(H_1)-(H_4)$, $(\widetilde{H%
}_5)$ and $(H_6)-(H_7)$ hold and let $c$ be defined in \eqref{e-c-def}. The
integral equation \eqref{e-Int} has at least one non trivial solution in $K$
if one of the following conditions hold

\begin{itemize}
\item[(C1)] There exist $\rho_1,\,\rho_2 \in (0,\infty)$, $\frac{\rho_1}{c}%
<\rho_2$, such that $(I^0_{\rho_1})$ and $(I^1_{\rho_2})$ are verified.

\item[(C2)] There exist $\rho_1,\,\rho_2 \in (0,\infty)$, $\rho_1<\rho_2$,
such that $(I^1_{\rho_1})$ and $(I^0_{\rho_2})$ are verified.
\end{itemize}

The integral equation \eqref{e-Int} has at least two non trivial solutions
in $K$ if one of the following conditions hold

\begin{itemize}
\item[(C3)] There exist $\rho_1,\,\rho_2,\,\rho_3 \in (0,\infty)$, $\frac{%
\rho_1}{c}<\rho_2<\rho_3$, such that $(I^0_{\rho_1})$, $(I^1_{\rho_2})$ and $%
(I^0_{\rho_3})$ are verified.

\item[(C4)] There exist $\rho_1,\,\rho_2,\,\rho_3 \in (0,\infty)$, with $%
\rho_1<\rho_2$ and $\frac{\rho_2}{c}<\rho_3$, such that $(I^1_{\rho_1})$, $%
(I^0_{\rho_2})$ and $(I^1_{\rho_3})$ are verified.
\end{itemize}

The integral equation \eqref{e-Int} has at least three non trivial solutions
in $K$ if one of the following conditions hold

\begin{itemize}
\item[(C5)] There exist $\rho_1,\,\rho_2,\,\rho_3,\,\rho_4 \in (0,\infty)$,
with $\frac{\rho_1}{c}<\rho_2<\rho_3$ and $\frac{\rho_3}{c}<\rho_4$, such
that $(I^0_{\rho_1})$, $(I^1_{\rho_2})$, $(I^0_{\rho_3})$ and $%
(I^1_{\rho_4}) $ are verified.

\item[(C6)] There exist $\rho_1,\,\rho_2,\,\rho_3,\,\rho_4 \in (0,\infty)$,
with $\rho_1<\rho_2$ and $\frac{\rho_2}{c}<\rho_3<\rho_4$, such that $%
(I^1_{\rho_1})$, $(I^0_{\rho_2})$, $(I^1_{\rho_3})$ and $(I^0_{\rho_4})$ are
verified.
\end{itemize}
\end{theorem}

The proof of the previous result is an immediate consequence of the
properties of the fixed point index. Moreover, it must be point out that,
despite of the fact that the previous theorem studies the existence of one,
two or three solutions, similar results can be formulated to ensure the
existence of four or more solutions.

\section{Examples}

\begin{example}
Consider the following boundary value problem: 
\begin{equation}  \label{ex1}
\left\{%
\begin{split}
& u^{(3)}(t)=\lambda\,\frac{e^t\left(|u(t)|+|u^{\prime }(t)|+|u^{\prime
\prime }(t)|\right)}{1+(u(t))^2}, \quad t\in [0,1], \\[2pt]
& u(0)=-u(1), \ u^{\prime }(0)=\frac{1}{2}\,u^{\prime }(1), \ u^{\prime
\prime }(0)=0.
\end{split}%
\right.
\end{equation}

The Green's function related to the homogeneous problem 
\begin{equation*}
\left\{%
\begin{split}
& u^{(3)}(t)=0, \quad t\in [0,1], \\
& u(0)=-u(1), \ u^{\prime }(0)=\frac{1}{2}\,u^{\prime }(1), \ u^{\prime
\prime }(0)=0,
\end{split}%
\right.
\end{equation*}
is the following one 
\begin{equation*}
G(t,s)=\left\{%
\begin{array}{ll}
\frac{1}{4}\,(1-s)\,(-3+s+4\,t), & \ t\le s, \\[4pt] 
\frac{1}{4}\,(-3+s\,(s+4)+2\,t\,(t+2)-8\,s\,t), & \ s<t.%
\end{array}
\right.
\end{equation*}
Therefore, solutions of the boundary value problem \eqref{ex1} correspond
with the fixed points of the following operator: 
\begin{equation*}
{\mathcal{T}}u(t)= \lambda \,\int_{0}^{1}G(t,s)\,f(s,u(s),u^{\prime
}(s),u^{\prime \prime }(s))\,\dif s,\quad t\in \lbrack 0,1],
\end{equation*}
which is a particular case of the operator defined in \eqref{e-Int} for $T=1$%
, $m=2$, $k\equiv G$ and $f(t,x,y,z)=\frac{e^t\left(|x|+|y|+|z|\right)}{1+x^2%
}$. We will check now that the kernel $G$ satisfies conditions $(H_1)-(H_5)$%
. To do that, we need to calculate the explicit expression of the first and
second partial derivatives of the Green's function, that is, 
\begin{equation*}
\frac{\partial\, G}{\partial \,t}(t,s)= \left\{%
\begin{array}{ll}
1-s, & \ t\le s, \\[4pt] 
1-2\,s+t, & \ s<t,%
\end{array}
\right.
\end{equation*}
and 
\begin{equation*}
\frac{\partial^2\, G}{\partial\, t^2}(t,s)=\left\{%
\begin{array}{ll}
0, & \ t< s, \\[4pt] 
1, & \ s<t.%
\end{array}
\right.
\end{equation*}
Using this expressions, we are able to check that the required conditions
hold:

\begin{itemize}
\item[$(H_1)$] Let $\tau \in I$ be fixed. Both $G$ and $\frac{\partial\, G}{\partial\,t}$ are uniformly continuous so the hypothesis is immediate for $i=0,1$. Moreover, for the second derivative $\frac{\partial^2\, G}{\partial\,t^2}$ (that is, for the case $i=m=2$), we can take $Z_\tau=\{\tau\}$ and we have that
\[\left|\frac{\partial^2\, G}{\partial\,t^2}(t,s)-\frac{\partial^2\, G}{\partial\,t^2}(\tau,s)\right|=|1-1|=0, \ \forall\, s<\min\{t,\tau\} \]
and 
\[\left|\frac{\partial^2\, G}{\partial\,t^2}(t,s)-\frac{\partial^2\, G}{\partial\,t^2}(\tau,s)\right|=|0-0|=0, \ \forall\, s>\max\{t,\tau\}, \]
so the hypothesis hold.
%

\item[$(H_2)$] Numerically, it can be seen that 
\begin{equation*}
G(t,s)\ge 0, \ \text{ for all } t\in[t_0,1], \ s\in[0,1],
\end{equation*}
with $t_0\approx 0.6133$. Therefore, in this case $[m_0,n_0]=[t_0,1]$.

Moreover, both $\frac{\partial\,G}{\partial\,t}$ and $\frac{\partial^2\,G}{%
\partial\,t^2}$ are nonnegative on the square $[0,1]\times [0,1]$, which
means that $[m_1,n_1]=[m_2,n_2]=[0,1]$.

\item[$(H_3)$] It can be checked that 
\begin{equation*}
\left|G(t,s) \right|\le \frac{1}{4}\,(3-4\,s+s^2), \ \text{ for all } t\in[%
0,1], \ s\in[0,1],
\end{equation*}
and the equality holds for $t=0$ and $t=1$ so the choice $h_0(s)=\frac{1}{4}%
\,(3-4\,s+s^2)$ is optimal. This inequality can be easily proved by taking
into account that, since $\frac{\partial\, G}{\partial\,t}$ is nonnegative,
then $G(\cdot,s)$ is nondecreasing for every $s\in [0,1]$ and, therefore, 
\begin{equation*}
|G(t,s)|\le \max\{|G(0,s)|,\,|G(1,s)| \}= \frac{1}{4}\, (3-4\,s+s^2).
\end{equation*}

For the first derivative, it holds that 
\begin{equation*}
\left|\frac{\partial\,G}{\partial\,t}(t,s) \right|\le 2\,(1-s), \ \text{ for
all } t\in[0,1], \ s\in[0,1],
\end{equation*}
and the equality holds for $t=1$, so $h_1(s)=2\,(1-s)$ is also optimal.

Finally, 
\begin{equation*}
\left|\frac{\partial^2\,G}{\partial\,t^2}(t,s) \right|\le 1, \ \text{ for } t%
\in[0,1] \text{ and a.\,e. }s\in[0,1],
\end{equation*}
and $h_2(s)=1$ is trivially optimal.

\item[$(H_4)$] If we take $\phi_0(s)=h_0(s)=\frac{1}{4}\,(3-4\,s+s^2)$, $%
[c_0,d_0]=[0,1]$, and $[a_0,b_0]=[t_1,1]$ with $t_1>t_0$ ($t_0$ given in $%
(H_2)$), it holds that there exists a constant $\xi_0(t_1)\in (0,1)$ such
that 
\begin{equation*}
G(t,s)\ge \xi_0(t_1)\,\phi_0(s), \ \text{ for all } t\in[t_1,1], \ s\in[0,1].
\end{equation*}
We note that the bigger $t_1$ is, the bigger the constant $\xi_0(t_1)$ is.
For instance, if we take $t_1=0.62$, we can choose $\xi_0=\frac{1}{75}$.

With regard to the first derivative of $G$, it satisfies that 
\begin{equation*}
\frac{\partial\,G}{\partial\,t}(t,s)\le 2\,(1-s), \ \text{ for all } t\in[0,1%
], \ s\in[0,1],
\end{equation*}
that is, we could take $\phi_1(s)=h_1(s)=2(1-s)$, $[c_1,d_1]=[0,1]$, $\xi_1=%
\frac{1}{2}$ and $[a_1,b_1]=[0,1]$.

Finally, for the second derivative of $G$, it does not exist a suitable
function $\phi_2$ and a constant $\xi_2$ for which the inequalities in $%
(H_4) $ hold.

As a consequence, we deduce that $J_1=\{0,\,1\}$.

Moreover, it is obvious that $\int_{a_i}^{b_i}\phi_i(s)\,\dif s>0$ for $%
i=0,1 $.

\item[$(H_5)$] It is immediately deduced from the proofs of the previous
conditions.
\end{itemize}

Moreover, the nonlinearity $f$ satisfies condition $(H_6)$.

We will work in the cone 
\begin{equation*}
K=\left\{ 
\begin{split}
u\in \mathcal{C}^{2}([0,1],\mathbb{R}):\ & u(t)\geq 0,\,t\in \lbrack t_0,1],
\ \ u^{\prime }(t),u^{\prime \prime }(t)\geq 0,\,t\in \lbrack 0,1]; \\[3pt]
& \displaystyle\min_{t\in \left[ t_1,1\right] }u(t)\geq \xi_{0}(t_1)\, \Vert
u\Vert _{\left[ 0,1\right]}, \ \ \displaystyle\min_{t\in \left[0,1\right]
}u^{\prime }(t)\geq \frac{1}{2} \, \Vert u^{\prime }\Vert _{\left[ 0,1\right]
}
\end{split}
\right\}.
\end{equation*}

With the notation introduced in Section 3, we obtain the following values
for the constants involved in Theorem \ref{T:exist1_orden_n}: 
\begin{equation*}
\Lambda^0=\frac{1}{3}, \quad \Lambda^1=1, \quad \Lambda^2=1,
\end{equation*}
and therefore 
\begin{equation*}
\bar{\Lambda}=3\,\max\left\{\Lambda^0,\,\Lambda^1,\,\Lambda^2\right\} = 3,
\end{equation*}
\begin{equation*}
\Lambda_0=\xi_0(t_1)\left(\frac{1}{3}-\frac{3}{4}\,t_1 +\frac{1}{2}\, t_1^2-%
\frac{1}{12}\, t_1^3\right), \quad \Lambda_1=\frac{1}{2},
\end{equation*}
and so 
\begin{equation*}
\underline{\Lambda}=\max\left\{\xi_0^2(t_1)\left(\frac{1}{3}-\frac{3}{4}%
\,t_1 +\frac{1}{2}\, t_1^2-\frac{1}{12}\, t_1^3\right),\,\frac{1}{4}\right\}.
\end{equation*}
We note that, since $\xi_0(t_1)\in (0,1)$,
\[\xi_0^2(t_1)\left(\frac{1}{3}-\frac{3}{4}%
\,t_1 +\frac{1}{2}\, t_1^2-\frac{1}{12}\, t_1^3\right) < \frac{1}{3}-\frac{3}{4}%
\,t_1 +\frac{1}{2}\, t_1^2-\frac{1}{12}\, t_1^3 \]
and it is easy to see that the right hand side of previous inequality decreases with $t_1$ and, in particular, it is always smaller than $\frac{1}{4}$. Thus,
\[\underline{\Lambda}=\frac{1}{4}\]
independently of the value of $t_1$.

On the other hand, we obtain the following values for the limits over the
nonlinearity $f$: 
\begin{equation*}
f_{0}=\liminf\limits_{|x|,|y|,|z|\rightarrow 0}\,\min_{t\in \lbrack 0,1]}%
\frac{e^{t}\left( |x|+|y|+|z|\right) }{(1+x^{2})\left( |x|+|y|+|z|\right) }%
=\lim\limits_{|x|,|y|,|z|\rightarrow 0}\,\frac{1}{(1+x^{2})}=1,
\end{equation*}%
\begin{equation*}
f^{\infty }=\limsup\limits_{|x|,|y|,|z|\rightarrow \infty }\,\max_{t\in
\lbrack 0,1]}\frac{e^{t}\left( |x|+|y|+|z|\right) }{(1+x^{2})\left(
|x|+|y|+|z|\right) }=\lim\limits_{|x|,|y|,|z|\rightarrow \infty }\frac{e}{%
(1+x^{2})}=0.
\end{equation*}

Therefore, from Corollary \ref{cor_exist1}, we deduce that for all $%
\lambda\in \left(4,\infty\right)$, $\mathcal{T}$
has at least a fixed point in the cone $K$, with independence of the choice of $t_1$. This fixed point is a nontrivial solution of problem \eqref{ex1}.

On the other hand, we will prove that it is not possible to apply Theorem %
\ref{T:exist2_orden_n} to this example. With the notation introduced in
Lemma \ref{L:index0-fpi}, we have that 
\begin{equation*}
f_{\rho}^0=\inf\left\{ \frac{e^t(|x|+|y|+|z|)}{\rho\,(x^2+1)} : \ t\in[t_1,1]
, \ x\in\left[0,\frac{\rho}{\xi_0(t_1)}\right], \ y\in[0,2\rho], \ z\in[%
0,\rho] \right\}=0
\end{equation*}
and 
\begin{equation*}
f_{\rho}^1=\inf\left\{\frac{e^t(|x|+|y|+|z|)}{\rho\,(x^2+1)} : \ t\in[0,1],
\ x\in\left[0,\frac{\rho}{\xi_0(t_1)}\right], \ y\in[0,2\rho], \ z\in[0,\rho]%
\right\}=0,
\end{equation*}
and therefore it does not exist any $\rho$ such that condition $(I^0_{\rho})$
holds. Thus Theorem \ref{T:exist2_orden_n} is not applicable to this example.
\end{example}

\begin{example}
\label{ex:2_Dirichlet1}Consider now the following Lidstone fourth order
problem: 
\begin{equation}
\left\{ 
\begin{split}
& u^{(4)}(t)=\lambda \,t\,\left( e^{u(t)}+(u^{\prime }(t))^{2}+(u^{\prime
\prime }(t))^{2}+(u^{\prime \prime \prime }(t))^{2}\right) ,\quad t\in
\lbrack 0,1], \\[2pt]
& u(0)=u(1)=u^{\prime \prime }(0)=u^{\prime \prime }(1)=0.
\end{split}%
\right.   \label{ex:2_Dirichlet}
\end{equation}%
Notice that fourth order differential equations with this type of boundary
conditions have been applied for the study of the bending of simply
supported elastic beams (see \cite{M+TG+AS, Timo}) or suspension bridges
(see \cite{Drabek, Lazer}). \newline
The Green's function related to the homogeneous problem 
\begin{equation*}
\left\{ 
\begin{split}
& u^{(4)}(t)=0,\quad t\in \lbrack 0,1], \\[2pt]
& u(0)=u(1)=u^{\prime \prime }(0)=u^{\prime \prime }(1)=0.
\end{split}%
\right. 
\end{equation*}%
has the following expression: 
\begin{equation*}
G(t,s)=\frac{1}{6}\left\{ 
\begin{array}{ll}
t\,(1-s)\,(2\,s-s^{2}-t^{2}), & t\leq s, \\[4pt]
s\,(1-t)\,(2\,t-t^{2}-s^{2}), & s<t,%
\end{array}%
\right. 
\end{equation*}%
which implies that the solutions of problem \eqref{ex:2_Dirichlet} coincide
with the fixed points of 
\begin{equation*}
{\mathcal{T}}u(t)=\lambda \,\int_{0}^{1}G(t,s)\,f(s,u(s),u^{\prime
}(s),u^{\prime \prime }(s),u^{\prime \prime \prime }(s))\,\dif s,\quad t\in
\lbrack 0,1].
\end{equation*}%
Previous operator is a particular case of \eqref{e-Int} for $T=1$, $m=3$, $%
k\equiv G$ and $f(t,x,y,z,w)=t\,\left( e^{x}+y^{2}+z^{2}+w^{2}\right) $.

Next, we will give the explicit expressions of the first, second and third
derivatives of the Green's function:

\begin{equation*}
\frac{\partial\,G}{\partial\,t}(t,s)=\frac{1}{6}\left\{%
\begin{array}{ll}
-(1-s)\,(-2\,s+s^2+3\,t^2), & t\le s, \\[4pt] 
s\,(2+s^2+3\,t^2-6\,t), & s<t,%
\end{array}
\right.
\end{equation*}

\begin{equation*}
\frac{\partial^2\,G}{\partial\,t^2}(t,s)=\left\{%
\begin{array}{ll}
-t\,(1-s), & t\le s, \\[4pt] 
-s\,(1-t), & s<t,%
\end{array}
\right.
\end{equation*}

\begin{equation*}
\frac{\partial^3\,G}{\partial\,t^3}(t,s)=\left\{%
\begin{array}{ll}
-(1-s), & t< s, \\[4pt] 
s, & s<t,%
\end{array}
\right.
\end{equation*}
and now we will see that they satisfy the required hypotheses:

\begin{itemize}
\item[$(H_1)$] As in previous example, it is easy to verify that this condition holds.


\item[$(H_2)$] The Green's function $G$ is nonnegative on $[0,1]\times[0,1]$
(in fact it is positive on $(0,1)\times(0,1)$). Therefore $[m_0,n_0]=[0,1]$.

For first derivative it holds that 
\begin{equation*}
\frac{\partial\,G}{\partial\,t}(t,s)\ge 0 \quad \text{for all } t\in[0,t_2],
\ s\in[0,1],
\end{equation*}
with $t_2=1-\frac{\sqrt{3}}{3}\approx 0.42265$. Thus $[m_1,n_1]=[0,t_2]$.

With respect to the second derivative, it is immediate to see that it is
nonpositive on its square of definition. However it is zero on the boundary
of the square, so we could take $[m_2,n_2]=\{0\}$ (it would also be possible
to choose $[m_2,n_2]=\{1\}$).

Finally, the third derivative is nonnegative on the triangle 
\begin{equation*}
\{(t,s)\in[0,1]\times[0,1] : \ s<t\},
\end{equation*}
that is, $[m_3,n_3]=\{1\}$.

\item[$(H_{3})$] We have that 
\begin{equation*}
\left\vert G(t,s)\right\vert =G(t,s)\leq h_{0}(s)\quad \text{for all }t\in
\lbrack 0,1],\ s\in \lbrack 0,1],
\end{equation*}%
where 
\begin{equation*}
h_{0}(s)=\frac{1}{9\,\sqrt{3}}\left\{ 
\begin{array}{ll}
s\,(1-s^{2})^{\frac{3}{2}}, & 0\leq s\leq \frac{1}{2}, \\[4pt] 
(1-s)\,(2\,s-s^{2})^{\frac{3}{2}}, & \frac{1}{2}<s\leq 1.%
\end{array}%
\right.
\end{equation*}%
Previous inequality has been proved in \cite{WIF}. 

%
%

%


Previous inequality is optimal in the sense that for each $s\in[0,1]$ there
exists at least one value of $t\in[0,1]$ for which the equality is satisfied.

Analogously, it holds that 
\begin{equation*}
\left|\frac{\partial\,G}{\partial\,t}(t,s)\right| \le h_1(s) \quad \text{for
all } t\in[0,1], \ s\in[0,1],
\end{equation*}
for 
\begin{equation*}
h_1(s)=\frac{1}{6}\,s\,(1-s)\left\{%
\begin{array}{ll}
2-s, & 0\le s\le \frac{1}{2}, \\[4pt] 
1+s, & \frac{1}{2} <s\le 1,%
\end{array}
\right.
\end{equation*}
and the equality holds for $0\le s\le \frac{1}{2}$ at $t=0$ and for $\frac{1}{2} <s\le 1$ at $t=1$, so this choice of $h_1$ is optimal. 

For the second derivative, we have that 
\begin{equation*}
\left|\frac{\partial^2\,G}{\partial\,t^2}(t,s)\right| \le s(1-s)\equiv
h_2(s) \quad \text{for all } t\in[0,1], \ s\in[0,1],
\end{equation*}
and the inequality is optimal in the same way than for the Green's function $%
G$.

With regard to the third derivative, it satisfies that 
\begin{equation*}
\left|\frac{\partial^3\,G}{\partial\,t^3}(t,s)\right| \le
\max\{s,1-s\}\equiv h_3(s) \quad \text{for } t\in[0,1] \text{ and a.\,e. } s%
\in[0,1],
\end{equation*}
and the inequality is also optimal.

\item[$(H_4)$] If we choose $\phi_0(s)=h_0(s)$, given in $(H_3)$, and $%
[c_0,d_0]=[0,1]$, then for any closed interval $[a_0,b_0]\subset (0,1)$, it
is possible to find a constant $\xi_0(a_0,b_0)\in (0,1)$ such that 
\begin{equation*}
G(t,s)\ge \xi_0(a_0,b_0)\,\phi_0(s), \ \text{ for all } t\in[a_0,b_0], \ s\in%
[0,1].
\end{equation*}
This has been proved in \cite{WIF} with an explicit function. Of course, it
is satisfied that the bigger the interval $[a_0,b_0]$ is, the smaller $%
\xi_0(a_0,b_0)$ needs to be.

Analogously, we can take $\phi_1(s)=h_1(s)$ and $[c_1,d_1]=[0,1]$ and it
holds that for any interval $[0,b_1]$, with $b_1< 1-\frac{\sqrt{3}}{3}$,
there exists $\xi_1(b_1)\in (0,1)$ such that 
\begin{equation*}
\frac{\partial\,G}{\partial\,t}(t,s)\ge \xi_1(b_1)\,\phi_1(s), \ \text{ for
all } t\in[0,b_1], \ s\in[0,1].
\end{equation*}

Finally, with respect to the second derivative of the Green's function $G$,
it does not exist any pair of function $\phi _{2}$ and constant $\xi _{2}$
such that the inequalities in $(H_{4})$ hold. The same occurs with the third
derivative of $G$. Therefore $J_{1}=\{0,\,1\}.$

\item[$(H_{5})$] It is a direct consequence of $(H_{2})$.

Clearly, $f$ satisfies $(H_{6})$ and $(H_{7})$.
\end{itemize}

As a consequence of the properties of the Green's function that we have just
seen, we will work in the cone 
\begin{equation*}
K=\left\{ 
\begin{split}
u\in \mathcal{C}^{3}([0,1],\mathbb{R}):\ & u(t)\geq 0,\,t\in \lbrack 0,1],\
\ u^{\prime }(t)\geq 0,\,t\in \lbrack 0,t_{2}], \\[4pt]
& u^{\prime \prime }(t)\geq 0,\ t\in \{0,1\},\ \ u^{\prime \prime \prime
}(1)\geq 0, \\[4pt]
& \displaystyle\min_{t\in \left[ a_{0},b_{0}\right] }u(t)\geq \xi
_{0}(a_{0},b_{0})\,\Vert u\Vert _{\left[ 0,1\right] }, \\[4pt]
& \displaystyle\min_{t\in \left[ 0,b_{1}\right] }u^{\prime }(t)\geq \xi
_{1}(b_{1})\,\Vert u^{\prime }\Vert _{\left[ 0,1\right] }
\end{split}%
\right\} .
\end{equation*}

Moreover, we will make all the calculations with the values $%
[a_{0},b_{0}]=[0.1,0.9]$, $\xi _{0}=\frac{1}{4}$, $[0,b_{1}]=\left[ 0,\frac{1%
}{3}\right] $ and $\xi _{1}=\frac{1}{6}$.

In this case, with the notation introduced in Lemma \ref{L:index1_fp}, we
have that 
\begin{equation*}
\frac{1}{N}=\max\left\{\frac{5}{384}, \, \frac{1}{24}, \, \frac{1}{8}, \, 
\frac{1}{2} \right\}=\frac{1}{2}
\end{equation*}
and 
\begin{equation*}
f^{\rho_2}=\sup\left\{\frac{t\left(e^x+y^2+z^2+w^2\right)}{\rho_2}: \ t\in[%
0,1], \ x,\,y,\,z,\,w \in [-\rho_2,\rho_2] \right\}= \frac{%
e^{\rho_2}+3\,\rho_2^2}{\rho_2},
\end{equation*}
and so $\left(I^1_{\rho_2}\right)$ holds for any $\lambda<\frac{2 \,\rho_2}{%
e^{\rho_2}+3\,\rho_2^2}$.

Analogously, with the notation used in Lemma \ref{L:index0-fpi}, 
\begin{equation*}
\frac{1}{M_0}=\frac{29}{7500}, \quad \frac{1}{M_1}=\frac{7}{1944},
\end{equation*}
\begin{equation*}
f_{\rho_1}^0=\inf\left\{\frac{t\left(e^x+y^2+z^2+w^2\right)}{\rho_1} : \ t\in%
[0.1,0.9], \ x\in\left[0,4\,\rho_1\right], \ y\in[0,6\,\rho_1], \ z,\,w\in[%
0,\rho_1] \right\}= \frac{0.1}{\rho_1}
\end{equation*}
and 
\begin{equation*}
f_{\rho}^1=\inf \left\{\frac{t\left(e^x+y^2+z^2+w^2\right)}{\rho_1} : \ t\in%
\left[0,\frac{1}{3}\right], \ x\in\left[0,4\,\rho_1\right], \ y\in[%
0,6\,\rho_1], \ z,\,w\in[0,\rho_1] \right\}=0,
\end{equation*}
and thus $\left(I^0_{\rho_1}\right)$ holds for $\lambda>\frac{75000 \,\rho_1%
}{29}$.

Therefore, as a consequence of $(C_1)$ in Theorem \ref{T:exist2_orden_n},
for any pair of values $\rho_1,\,\rho_2>0$ such that $\rho_1<c\,\rho_2=\frac{%
\rho_2}{6}$ and 
\begin{equation*}
\frac{75000 \,\rho_1}{29}<\frac{2 \,\rho_2}{e^{\rho_2}+3\,\rho_2^2},
\end{equation*}
problem \eqref{ex:2_Dirichlet} has at least a nontrivial solution for all 
\begin{equation*}
\lambda\in \left(\frac{75000 \,\rho_1}{29}, \, \frac{2 \,\rho_2}{%
e^{\rho_2}+3\,\rho_2^2}\right).
\end{equation*}
In particular, there exists at least a nontrivial solution of %
\eqref{ex:2_Dirichlet} for all 
\begin{equation*}
\lambda\in \left(0, \,0.4171\right).
\end{equation*}

\vspace*{2pt}

On the other hand, we obtain that: 
\begin{equation*}
f_{0}=\liminf\limits_{|x|,|y|,|z|,|w|\rightarrow 0}\,\min_{t\in \lbrack 0,1]}%
\frac{t\left( e^{x}+y^{2}+z^{2}+w^{2}\right) }{|x|+|y|+|z|+|w|}=0,
\end{equation*}%
and thus neither Theorem \ref{T:exist1_orden_n} nor Corollary \ref%
{cor_exist1} can be applied to this example.
\end{example}

\section{Application to some $2n$-$th$ order problems}

In this section we contribute to fill some gaps on the study of general $2n$-%
$th$ order Lidstone boundary value problems, for $n\geq 1,$ usually the
nonlinearities may depend only on the even derivatives (see, for example, 
\cite{Davis, Marius, Wang, Z+L}), or general complementary Lidstone
problems, (see \cite{Wong} and the references therein). Therefore, we
consider the following problem, with a full nonlinearity, 
\begin{equation}
\left\{ 
\begin{split}
& u^{(2n)}(t)=f\left( t,u(t),\dots ,u^{(2n-1)}(t)\right) ,\quad t\in \lbrack
0,1], \\[2pt]
& u^{(2k)}(0)=u^{(2k)}(1)=0,\ \ k=0,\dots ,n-1.
\end{split}%
\right.  \label{ex:order_n}
\end{equation}

Let $G(t,s)$ be the Green's function related to the homogeneous problem 
\begin{equation*}
\left\{ 
\begin{split}
& u^{(2n)}(t)=0,\quad t\in \lbrack 0,1], \\[2pt]
& u^{(2k)}(0)=u^{(2k)}(1)=0,\ \ k=0,\dots ,n-1.
\end{split}%
\right.
\end{equation*}

It can be checked that, for $n\ge 2$, $g(t,s)=\frac{\partial ^{2n-4}\,G}{%
\partial \,t^{2n-4}}(t,s)$ is the Green's function related to the problem 
\begin{equation*}
\left\{%
\begin{split}
& u^{(4)}(t)=0,\quad t\in \lbrack 0,1], \\[2pt]
& u(0)=u(1)=u^{\prime \prime }(0)=u^{\prime \prime }(1)=0,
\end{split}%
\right.
\end{equation*}
whose explicit expression has been calculated in Example \ref%
{ex:2_Dirichlet1}. As a consequence of the calculations made in that example
we know that the following facts hold, for $n\ge 2$:

\begin{itemize}
\item $\frac{\partial^{2n-4} \, G}{\partial \, t^{2n-4}}(t,s)=g(t,s) \ge 0$
on $[0,1]\times [0,1]$ and $\frac{\partial^{2n-4} \, G}{\partial \, t^{2n-4}}%
(t,s)=0$ on the boundary of the square.

\item $\frac{\partial^{2n-3} \, G}{\partial \, t^{2n-3}}(t,s)=\frac{%
\partial\,g}{\partial\,t}(t,s) \ge 0$ on $[0,t_2]\times [0,1]$, with $t_2=1-%
\frac{\sqrt{3}}{3}$.

\item $\frac{\partial^{2n-2} \, G}{\partial \, t^{2n-2}}(t,s)=\frac{%
\partial^2\,g}{\partial\,t^2}(t,s) \le 0$ on $[0,1]\times [0,1]$, and $\frac{%
\partial^{2n-2} \, G}{\partial \, t^{2n-2}}(t,s)=0$ on the boundary of the
square.

\item $\frac{\partial^{2n-1} \, G}{\partial \, t^{2n-1}}(t,s)=\frac{%
\partial^3\,g}{\partial\,t^3}(1,s) \ge 0$ for $s\in[0,1]$.
\end{itemize}

With this information, we can obtain some results about the constant sign
both of the derivatives of smaller order of $G$ and of the Green's function
itself.

\begin{enumerate}
\item Since $\frac{\partial ^{2n-4}\,G}{\partial \,t^{2n-4}}(t,s)\geq 0$,
for $n\geq 3$, it holds that for each fixed $s\in \lbrack 0,1]$, $\frac{%
\partial ^{2n-5}\,G}{\partial \,t^{2n-5}}(\cdot ,s)$ is increasing.

Assume that it is nonnegative. Then it would occur that $\frac{\partial
^{2n-6}\,G}{\partial \,t^{2n-6}}(\cdot ,s)$ is also increasing and, since
from the boundary value conditions it holds that $\frac{\partial ^{2n-6}\,G}{%
\partial \,t^{2n-6}}(0,s)=\frac{\partial ^{2n-6}\,G}{\partial \,t^{2n-6}}%
(1,s)=0$, we would conclude that $\frac{\partial ^{2n-6}\,G}{\partial
\,t^{2n-6}}(t,s)=0$ on $[0,1]\times \lbrack 0,1]$, which is not possible.

The same argument holds if we assume that $\frac{\partial^{2n-5} \, G}{%
\partial \, t^{2n-5}}(\cdot,s)$ is nonpositive.

Therefore, necessarily $\frac{\partial^{2n-5} \, G}{\partial \, t^{2n-5}}%
(\cdot,s)$ is sign-changing and, since it is increasing, we know for sure
that $\frac{\partial^{2n-5} \, G}{\partial \, t^{2n-5}}(0,s)<0$ and $\frac{%
\partial^{2n-5} \, G}{\partial \, t^{2n-5}}(1,s)>0$ for all $s\in [0,1]$.

\item Now, since $\frac{\partial^{2n-5} \, G}{\partial \, t^{2n-5}}(\cdot,s)$
is sign-changing and increasing, $\frac{\partial^{2n-6} \, G}{\partial \,
t^{2n-6}}(\cdot,s)$ will be first decreasing and then increasing. This
together with the boundary value conditions $\frac{\partial^{2n-6} \, G}{%
\partial \, t^{2n-6}}(0,s)=\frac{\partial^{2n-6} \, G}{\partial \, t^{2n-6}}%
(1,s)=0$ implies that $\frac{\partial^{2n-6} \, G}{\partial \, t^{2n-6}}$ is
nonpositive.

\item Since $\frac{\partial ^{2n-6}\,G}{\partial \,t^{2n-6}}$ is
nonpositive, we can follow an analogous argument to the one made in 1. to
deduce that $\frac{\partial ^{2n-7}\,G}{\partial \,t^{2n-7}}$ is
sign-changing and decreasing. In particular this implies that $\frac{%
\partial ^{2n-7}\,G}{\partial \,t^{2n-7}}(0,s)>0$ and $\frac{\partial
^{2n-7}\,G}{\partial \,t^{2n-7}}(1,s)<0$ for all $s\in \lbrack 0,1]$ and $%
n\geq 4$.

\item Finally, arguing analogously to 2., we can deduce that $\frac{\partial
^{2n-8}\,G}{\partial \,t^{2n-8}}$ is nonnegative on $[0,1]\times \lbrack
0,1],$ for $n\geq 4$.
\end{enumerate}

We note that we could repeat all the previous arguments iteratively and this
way we could deduce the following sign-criteria for the derivatives of $G$.
So, for $n\geq \frac{k}{2}:$

\begin{itemize}
\item If $k\equiv 0\left( \mod{4}\right) $, then $\frac{\partial ^{2n-k}\,G}{%
\partial \,t^{2n-k}}(t,s)\geq 0$ on $[0,1]\times \lbrack 0,1]$.

\item If $k\equiv 1\left( \mod{4}\right) $, then $\frac{\partial ^{2n-k}\,G}{%
\partial \,t^{2n-k}}(\cdot ,s)$ is sign-changing and increasing for every $%
s\in \lbrack 0,1]$. In particular, $\frac{\partial ^{2n-k}\,G}{\partial
\,t^{2n-k}}(0,s)<0$ and $\frac{\partial ^{2n-k}\,G}{\partial \,t^{2n-k}}%
(1,s)>0$ for every $s\in \lbrack 0,1]$.

\item If $k\equiv 2\left( \mod{4}\right) $, then $\frac{\partial ^{2n-k}\,G}{%
\partial \,t^{2n-k}}(t,s)\leq 0$ on $[0,1]\times \lbrack 0,1]$.

\item If $k\equiv 3\left( \mod{4}\right) $, then $\frac{\partial ^{2n-k}\,G}{%
\partial \,t^{2n-k}}(\cdot ,s)$ is sign-changing and decreasing for every $%
s\in \lbrack 0,1]$. In particular, $\frac{\partial ^{2n-k}\,G}{\partial
\,t^{2n-k}}(0,s)>0$ and $\frac{\partial ^{2n-k}\,G}{\partial \,t^{2n-k}}%
(1,s)<0$ for every $s\in \lbrack 0,1]$.
\end{itemize}

In particular, if $n$ is even, we could deduce that $G(t,s)\geq 0$ on $%
[0,1]\times \lbrack 0,1]$ and, for $n$ odd, $G(t,s)\leq 0$ on $[0,1]\times
\lbrack 0,1]$.

Therefore, the Green's function and its derivatives satisfy the required
hypotheses:

\begin{itemize}
\item[$(H_{1})$] As in Example \ref{ex:2_Dirichlet1}, this condition holds
as a direct consequence of the general properties of the Green's function.

\item[$(H_{2})$] As we have just proved, we could take $%
[m_{2n-i},n_{2n-i}]=[0,1]$ for $i\equiv 0\left( \mod{4}\right) $, $%
[m_{2n-i},n_{2n-i}]=\{1\}$ for $i\equiv 1\left( \mod{4}\right) $, $%
[m_{2n-i},n_{2n-i}]=\{0\}$ for $i\equiv 2\left( \mod{4}\right) $ and $%
[m_{2n-i},n_{2n-i}]=\{0\}$ for $i\equiv 3\left( \mod{4}\right) $.

\item[$(H_{3})$] It is enough to take $h_{i}(s)=\max \left\{ \left\vert 
\frac{\partial ^{i}\,G}{\partial \,t^{i}}(t,s)\right\vert :\ t\in \lbrack
0,1]\right\} ,$ for $i\in J$.

\item[$(H_{4})$] For $n\geq 2,$ we could take $J_{1}=\{2n-4,2n-3\}$. As a
consequence of Example \ref{ex:2_Dirichlet1}, we know that 
\begin{equation*}
\left\vert \frac{\partial ^{2n-4}\,G}{\partial \,t^{2n-4}}(t,s)\right\vert =%
\frac{\partial ^{2n-4}\,G}{\partial \,t^{2n-4}}(t,s)=g(t,s)\leq \phi
_{2n-4}(s),
\end{equation*}%
with 
\begin{equation*}
\phi _{2n-4}(s)=\frac{1}{9\,\sqrt{3}}\left\{ 
\begin{array}{ll}
s\,(1-s^{2})^{\frac{3}{2}}, & 0\leq s\leq \frac{1}{2}, \\[4pt] 
(1-s)\,(2\,s-s^{2})^{\frac{3}{2}}, & \frac{1}{2}<s\leq 1.%
\end{array}%
\right.
\end{equation*}%
Moreover, it holds that for any closed interval $[a_{2n-4},b_{2n-4}]\subset
\lbrack 0,1]$, there exists a constant $\xi _{2n-4}(a_{2n-4},b_{2n-4})\in
(0,1)$ such that 
\begin{equation*}
\frac{\partial ^{2n-4}\,G}{\partial \,t^{2n-4}}(t,s)\geq \xi
_{2n-4}(a_{2n-4},b_{2n-4})\,\phi _{2n-4}(s),\ \text{ for all }t\in \lbrack
a_{2n-4},b_{2n-4}],\ s\in \lbrack 0,1].
\end{equation*}

Analogously, from Example \ref{ex:2_Dirichlet1} we know that 
\begin{equation*}
\left\vert \frac{\partial ^{2n-3}\,G}{\partial \,t^{2n-3}}(t,s)\right\vert
=\left\vert \frac{\partial \,g}{\partial \,t}(t,s)\right\vert \leq \phi
_{2n-3}(s)=\frac{1}{6}\,s\,(1-s)\left\{ 
\begin{array}{ll}
2-s, & 0\leq s\leq \frac{1}{2}, \\[4pt] 
1+s, & \frac{1}{2}<s\leq 1,%
\end{array}%
\right.
\end{equation*}%
for all $t\in \lbrack 0,1]$, $s\in \lbrack 0,1]$ and for any interval $%
[0,b_{2n-3}]$, with $b_{2n-3}<1-\frac{\sqrt{3}}{3}$ there exists $\xi
_{2n-3}(b_{2n_{3}})\in (0,1)$ such that 
\begin{equation*}
\frac{\partial ^{2n-3}\,G}{\partial \,t^{2n-3}}(t,s)\geq \xi
_{2n-3}(b_{2n_{3}})\,\phi _{2n-3}(s),\ \text{ for all }t\in \lbrack
0,b_{2n-3}],\ s\in \lbrack 0,1].
\end{equation*}

\item[$(H_5)$] As we have already seen, it holds that $%
[m_{2n-4},n_{2n-4}]=[0,1]$.
\end{itemize}

Then, we could work in the cone, for $n\in 
\mathbb{N}
$ such that $n\geq \max \left\{ 2,\frac{i}{2}\right\} ,$ 
\begin{equation*}
K=\left\{ 
\begin{split}
u\in \mathcal{C}^{2n-1}([0,1],\mathbb{R}):\ & u^{(2n-i)}(t)\geq 0,\,t\in
\lbrack 0,1],\ i\equiv 0\mod{4}, \\[4pt]
& u^{(2n-i)}(1)\geq 0,\ i\equiv 1\mod{4}, \\[4pt]
& u^{(2n-i)}(0)\geq 0,\ i\equiv 2\mod{4}, \\[4pt]
& u^{(2n-i)}(0)\geq 0,\ i\equiv 3\mod{4} \\[4pt]
& \displaystyle\min_{t\in \left[ a_{2n-4},b_{2n-4}\right] }u^{(2n-4)}(t)\geq
\xi _{2n-4}(a_{2n-4},b_{2n-4})\,\Vert u^{(2n-4)}\Vert _{\left[ 0,1\right] },
\\[4pt]
& \displaystyle\min_{t\in \left[ 0,b_{2n-3}\right] }u^{(2n-3)}(t)\geq \xi
_{2n-3}(b_{2n-3})\,\Vert u^{(2n-3)}\Vert _{\left[ 0,1\right] }
\end{split}%
\right\} .
\end{equation*}%
Thus, for any nonlinearity $f$ satisfying $(H_{6})$ and either conditions of
Theorem \ref{T:exist1_orden_n} or of Theorem \ref{T:exist2_orden_n}, it is
possible to find nontrivial solution of problem \eqref{ex:order_n}.

\section*{Acknowledgement}

This paper was mostly written during a stay of Luc\'{\i}a L\'{o}pez-Somoza
in \'{E}vora. Luc\'{\i}a L\'{o}pez-Somoza would like to acknowledge her
gratitude towards the Department of Mathematics of the University of \'{E}%
vora, and specially towards Professor Feliz Minh\'{o}s for his kindness and
hospitality.

\end{document}